\documentclass{amsart}

\usepackage{amsmath,amsfonts,amssymb}
\usepackage[latin1]{inputenc}
\usepackage{amscd}
\usepackage{latexsym}
\usepackage{bm}
\usepackage[all,cmtip]{xy}
\usepackage{euscript}
\usepackage{dsfont}
\usepackage{graphics}
\usepackage{mathrsfs}
\usepackage{graphicx}
\usepackage{tikz}

\newcommand{\N}{\mathbb{N}}
\newcommand{\R}{\mathbb{R}}

\newcommand{\di}{\displaystyle}
\newcommand{\T}{\mathbb{T}}

\usepackage{amsmath,amssymb,graphicx,mathrsfs,mathtools}
\usepackage[all]{xy}
\newcommand{\C}{\mathbf{C}}
\newcommand{\Q}{\mathbf{Q}}
\renewcommand{\P}{\mathscr{P}}
\renewcommand{\L}{\mathscr{L}}
\renewcommand{\l}{\ell}
\newcommand{\w}{{\rm w}}
\newcommand{\emb}{{\rm Emb}}
\newtheorem{theorem}{Theorem}
\newtheorem{definition}{Definition}
\newtheorem{lemma}{Lemma}

\begin{document}
\title[High-order discrete differential and integral calculus]{High-order discrete differential and integral calculus and Galerkin variational integrators}

\author{Jacky Cresson}
\address{Laboratoire de math\'ematiques et leurs applications, UMR CNRS 5142, Universit\'e de Pau et des Pays de l'Adour-E2S.}

\author{Khaled Hariz Belgacem}
\address{Department of Mathematics, University of Paderborn, Warburger Strasse 100, 33098 Paderborn, Germany.}

\author{Anna Szafranska}
\address{Institute of Applied Mathematics, Gdansk University of Technology, G. Narutowicz Street 11/12, 80-233 Gdansk, Poland.}

\begin{abstract}
We define high-order discrete differential and integral calculus following the discrete embedding formalism strategy. Using this framework, we define high-order discrete Lagrangian functionals and the corresponding calculus of variations. This provides a new derivation of Galerkin variational integrators.     
\end{abstract}

\maketitle


\tableofcontents

\section{Introduction}

Variational integrators are a class of geometric numerical integrators constructed using a discrete version of the Hamilton's principle. A comprehensive review is contained in the seminal article of Marsden and West \cite{marsden} and the classical book of Hairer and al. \cite{hairer}. A natural problem is to design high order variational integrators called {\it Galerkin variational integrators}. This was done in particular by Marsden and West \cite{marsden} and Leok \cite{leok1,leok2}. The convergence of these integrators is studied by S. Ober-Bl\"obaum and M. Vermeeren in \cite{ober}.\\

All the previous work deal with a derivation of the variational integrators initiated by Marsden and West \cite{marsden}. This derivation was analysed in \cite{rouba} and \cite{cr2}. Following the discrete embedding strategy \cite{cr1,cr2}, we propose to derive these integrators thanks to the definition of a discrete differential and integral calculus. However, this method was up to now restricted to order $1$ methods \cite{cr2,rouba} and extended to mid-point variational integrators in \cite{rouba2}.\\

In this paper, we develop a high-order discrete differential and integral calculus. This formalism allows us to define a natural discrete analogue of Lagrangian functionals. We then characterize the set of critical points of these functionals. This framework gives an alternative derivation of the Galerkin variational integrators.

\section{Embedding map on $\P^{m}_{\T} ([a,b])$}

 Let $N\in \mathbb{N}^*$, $a,b \in \R$ with $a < b$. We denote by $\T$ a uniform discrete time-scale defined by $\T =\{t_i \}_{i=0,\ldots, N}$ with $t_0 =a$ and $t_N =b$, $h=(b-a)/N$, $t_i =a+i h$, $i=0,\dots ,N$. As usual, we define $\T^-=\T\setminus \{a\}$, $\T^+=\T\setminus \{b\}$ and $\T^{\pm}=\T^+ \cap \T^- $. \\

Let $m\in \N^*$ be fixed. We denote by $\P^{m}_{\T} ([a,b])$ the set of piecewise continuous functions which are polynomials of degree $m$ over each interval $[t_i,t_{i+1} ]$, $i=0,\ldots,N-1$.\\

Let $\C =\{ c_1 ,\dots ,c_{m-1} \}$ be a discrete time-scale on $]0,1[$ such that $0<c_1 <\dots <c_{m-1} <1$. For all $i=0,\dots ,N-1$, we denote by $(\T )_{\C ,i}$ the discrete time-scale over $]t_i ,t_{i+1} [$ defined by $(\T )_{\C,i} =\{ c_{i,j} =t_i +c_i h \}_{j=1,\dots ,m-1}$. We denote by $\bar{\T}_{\C}$ the time-scale on $[a,b]$ defined by $\bar{\T}_{\C} =\T \cup (\T )_{\C}$ where $(\T)_{\C} =\di\sum_{i=1}^{m-1} (\T)_{\C,i}$,\\

Let $X\in C(\T ,\R )$. A function $\bar{X}\in C( \bar{\T}_{\C} ,\R )$ is called an extension of $X$ with respect to $\C$ if $\bar{X}|_{\T} :=X$. The function $X_C:= \bar{X}|_C$ is called the {\it control} function of $\bar{X}$.\\ 

For $i=0,\dots ,N-1$, we denote by $\bar{\mathbf{\l}}_i (t)= (\mathbf{\l}_i (t) ,\mathbf{\l}_{C,i} (t))$ the vector defined by 
\begin{equation}
\mathbf{\l}_i (t)=(\l_{i,0} (t) ,\l_{i,m} (t) ),\ \ \mathbf{\l}_{C,i} (t) =(\l_{i,1} (t),\dots ,\l_{i,m-1} (t)) ,
\end{equation}
where $\l_{i,0},\dots ,\l_{i,m}$ are the {\it Lagrange polynomials} over $(\T )_{C,i} \cup \{ t_{i+1} ,t_i \}$. \\

We also denote by $\mathbf{\bar{X}}_i =(\mathbf{X}_i ,\mathbf{X}_{C,i} )$ the vector defined by
\begin{equation}
\mathbf{X}_i = (X(t_i ), X (t_{i+1} )),\ \ \mathbf{X}_{C,i} =(\bar{X} (c_{i,1} ) ,\dots , \bar{X} (c_{i,m-1} ) ) ,
\end{equation}
$i=0,\dots ,N-1$.\\

The embedding map $\emb_m :C(\bar{\T}_{\C} ,\R ) \rightarrow \P^m_{\T} ([a,b])$ is defined for all $t\in [t_i ,t_{i+1}]$ by
\begin{equation}
\label{interdef}
    \emb_m (\bar{X} ) (t) := \langle \mathbf{\bar{X}}_i ,\bar{\mathbf{\l}}_i (t) \rangle , 
\end{equation}
where $\langle \cdot ,\cdot \rangle$ is the usual scalar product on $\R^{m+1}$,\\ 

We denote by $\bar{x} \in \P^m_{\T} ([a,b])$ the function
\begin{equation}
    \bar{x} = \emb_m (\bar{X} ) ,
\end{equation}
Denoting  by $x$ and $x_C$ the functions 
\begin{equation}
    x:= \langle \mathbf{X}_i ,\mathbf{\l}_i (t) \rangle\ {\rm and}\ x_C :=\langle \mathbf{X}_{C,i} ,\mathbf{\l}_{C,i} (t) \rangle ,
\end{equation}
we obtain
\begin{equation}
\label{decomp}
\bar{x} =x+x_C .
\end{equation}

By construction, the embedding map $\emb_m$ is a {\it linear} mapping on $C(\bar{\T}_{\C} ,\R )$, i.e. for all $\alpha\in \R$, $\bar{X}$ and $\bar{Y}$ in $C(\bar{\T}_{\C} ,\R )$ we have 
\begin{equation}
    \emb_m (\alpha \bar{X} +\bar{Y} ) =\alpha \emb_m (\bar{X} ) +\emb_m (\bar{Y} ) ,
\end{equation}
or equivalently denoting by $\overline{x+y}$ the embedding of $\bar{X} +\bar{Y} $, we have
\begin{equation}
\label{lineinter}
    \overline{x+y} =\bar{x} +\bar{y} ,
\end{equation}

\section{High-order discrete differential and integral calculus}

\subsection{High-order discrete derivative} 

The embedding map of order $m$ can be used to defined for all $\bar{X}\in C(\bar{\T}_{\C} ,\R )$ a discrete derivative of order $m$:

\begin{definition}[Discrete derivative of order $m$]
For all $\bar{X} \in C(\bar{\T}_{\C} ,\R )$ ,we denote by $\Delta_m : C(\bar{\T}_{\C} ,\R ) \rightarrow C(\bar{T}_{\C}^+ ,\R )$ the mapping defined by
\begin{equation}
    \Delta_m (\bar{X} ) := \pi_{\T^+ } \left ( \di\frac{d^+}{dt} (\bar{x} ) \right ) ,
\end{equation}
where $\Pi_{\T^+} \left ( \di\frac{d^+}{dt} (\bar{x} ) \right )$ is the restriction of $\left ( \di\frac{d^+}{dt} (\bar{x} ) \right )$ to $\T^+$.
\end{definition}

For $m=1$, $\C =\emptyset$ so that $\bar{\T}_{\C} =\T$, $\bar{X} =X$ and $\Delta_1$ is the forward discrete derivative given by $\Delta_1 (X ) (t_i ) =\di\frac{1}{h} (X(t_{i+1} )-X(t_i ) )$, $i=0,\dots ,N-1$.\\ 

The linearity of $\emb_m$ implies that the operator $\Delta_m$ is a linear operator on $C(\bar{\T}_{\C} ,\R )$, i.e. for all $\alpha\in \R$, $\bar{X}$ and $\bar{Y}$ in $C(\bar{\T}_{\C} ,\R)$, we have 
\begin{equation}
    \Delta_m (\alpha \bar{X} +\bar{Y} )=\alpha \Delta_m (\bar{X})+\Delta_m (\bar{Y}) .
\end{equation}

\subsection{High-order discrete anti-derivatives}

We denote by $\Q =\{ q_0 ,q_1 ,\dots ,q_{\l } \}$ be a discrete time-scale of $[0,1]$ such that $q_0=0$, $q_{\l} =1$, $0<q_1 <\dots <q_{\l-1} <1$. A {\it weight} function on $\Q$ is a function $\w :\Q \rightarrow \R$ such that 
\begin{equation}
    \di\sum_{i=0}^{\l} \w (q_i )=1.
\end{equation}
The couple $(\Q ,\w )$ is called a {\it quadrature}.\\

Let $\Q$ be given. For all $i=0,\dots ,N-1$, we denote by $(\T )_{\Q}$ the discrete time-scale on $]t_i,t_{i+1}[$ defined by $(\T )_{\Q,i} =\{ q_{i,j}=t_i +q_i h \}_{j=1,\dots ,\l-1}$. We denote $\bar{\T}_{\Q}$ the time-scale on $[a,b]$ defined by $\bar{\T }_{\Q} =\T \cup (\T )_{\Q}$ where $(\T )_{\Q} :=\di\bigcup_{i=1}^{\l-1} (\T)_{\Q,i}$ \\

The construction of the anti-derivative is done as follows: 

\begin{definition}[Anti-derivative of order $q$] Let $\bar{f} \in C(\bar{\T}_{\Q} ;\R )$, and $f:=\pi (\bar{f} ) \in C(\T ,\R )$. The anti-derivative of $\bar{f}$ with respect to the quadrature $(\Q ,\w )$ is defined for all $t_i \in \T$, $i=0,\dots ,N$, by 
\begin{equation}
    \di\int_a^{t_i} \bar{f} (t)\, \Delta_{\w } t := \di\sum_{k=0}^{i-1} \left ( \w (0) f(t_k) +\w (1)f(t_{k+1} )  +\di\sum_{j=1}^{\ell -1} \w (q_j ) \bar{f} (q_{k,j} ) \right ) .
\end{equation}
\end{definition}

The anti-derivative of order $\ell$ is a linear operator, i.e. for all $\alpha \in \R$, $\bar{f}$ and $\bar{g}$ in $C(\T_{\Q} ,\R )$, and all $t_i \in \T$, we have 
\begin{equation}
    \di\int_a^{t_i}  \left ( \alpha \bar{f} +\bar{g} \right ) \, \Delta_{\w} t =
    \alpha \di\int_a^{t_i} \bar{f}\, \Delta_{\w} t +\di\int_a^{t_i} \, \bar{g}\, \Delta_{\w}t. 
\end{equation}

\subsection{A discrete Du Bois-Reymond lemma}

We derive in our framework, a discrete version of the Du Bois-Reymond lemmma of  calculus of variations (see \cite{gelfand}).\\

We denote by $C_{(0)} (\bar{\T}_{\C} ,\R )$ be a subset of  $C (\bar{\T}_{\C}, \R )$ defined by  
\begin{equation}
C_{(0)} (\bar{\T}_{\C} ,\R ) =\left\{ \bar{Y}\in C (\bar{\T}_{\C} ,\R ) ,\  \bar{Y}(a)=\bar{Y}(b)=0 \right\}.
\end{equation}

\begin{lemma}[Discrete Du Bois-Reymond Lemma]
\label{dubois}
Let $\bar{f}$ and $\bar{g}$ be in $C(\bar{T}_{\Q} ,\R )$ such that 
\begin{equation}
    \di\int_a^b \left ( \bar{f}(t) \bar{y} (t) +\bar{g} (t)\di\frac{d^+}{dt} (\bar{y} ) (t)\right ) \, \Delta_{\w} t =0 ,
\end{equation}
for all $\bar{Y} \in C_{(0)} (\bar{\T}_{\C} ,\R )$, $\bar{y}=\emb_m (\bar{Y} ) $. Then, we have for all $i=1,\dots ,N-1$:
\begin{equation}
\label{cond1}
\left .
\begin{array}{l}
\di\int_{t_{i-1}}^{t_i} \left ( 
    \bar{f} (t) \l_{i-1,m} (t)  
    +\bar{g} (t) \di\frac{d^+}{dt} (\l_{i-1 ,m}) 
    (t)\right ) \, \Delta_{\w} t \\
    +
    \di\int_{t_i}^{t_{i+1}} \left ( 
    \bar{f}(t) \l_{i,0} (t)  
    +\bar{g}(t)\di\frac{d^+}{dt} (\l_{i,0}) (t)
    \right ) \, \Delta_{\w} t 
    =0 .
\end{array}
\right .
\end{equation}
and
\begin{equation}
\label{cond2}
    \di\int_{t_i}^{t_{i+1}} \left ( 
    \bar{f}(t) \l_{i,j} (t)  
    +\bar{g}(t) \l'_{i,j} (t)  
    \right ) \, \Delta_{\w} t =0 ,\ \ j=1,\dots ,m-1.
\end{equation}
\end{lemma}

The proof is given in Section \ref{proofdubois}.

\section{High order discrete calculus of variations}

The previous formalism can be used to define a high-order discrete analogue of continuous Lagrangian functionals. We consider a Lagrangian functional of the form 
\begin{equation}
    \L (x) := \di\int_a^b L (x(t) ,\dot{x} (t) )\, dt ,
\end{equation}
with $x\in C^1 ([a,b],\R )$ and $L$ a real valued $C^2$ function defined over $\R \times \R$. \\

Let $(\Q ,\w )$ be a quadrature. For all $\bar{X} \in C(\bar{\T} ,\R )$, we define the {\it discrete $(\Q ,\w)$-Lagrangian functional of order $m$} as 
\begin{equation}
\label{dislag}
    \L (\bar{X} ) :=\di\int_a^b L (\bar{x} ,\di\frac{d^+}{dt} (\bar{x} )) \, \Delta_{\w} t 
\end{equation}

In the case $m=1$ and $q=1$, meaning without control points and quadrature points, we recover classical variational integrators depending on the quadrature which is chosen. In particular,  taking $w (0)=1$ and $w (1 )=0$, the functional reduces to the classical discrete Lagrangian functional 
\begin{equation}
\L ( X ) =h\di\sum_{i=0}^{N-1} L (X (t_i) ,\Delta_1 (X) (t_i) )  ,
\end{equation}
and for $w (0)=w (1)=1/2$, we obtain the mid-point discrete Lagrangian functional
\begin{equation}
\L (X)=
	 h\sum_{i=0}^{N-1}L\left(\frac{X(t_i)+X (t_{i+1})}{2},\Delta_1 (X) (t_i)\right) .
\end{equation}

The Fr\'echet derivative of  $\L$ at $\bar{X}$ along a  direction $\bar{Y} \in C_{(0)} (\bar{\T} ,\R )$ is defined  by
\begin{equation}
D\L (\bar{X} ) (\bar{Y})=\di\lim_{\epsilon\to 0}\frac{1}{\epsilon} (\L (\bar{X} +\epsilon \bar{Y} ) -\L (\bar{X} ) ) .
\end{equation}

A {\it critical point} $\bar{X}$ of $\L$ satisfies $D\L (\bar{X} ) (\bar{Y}) =0$ for all $\bar{Y} \in C_{(0)} (\bar{\T} ,\R )$. \\

As in the classical continuous case (see \cite{arnold}), we can characterize the critical points:

\begin{theorem}[Galerkin variational integrators]
\label{main}
Let $(\Q ,\w)$ be a quadrature. The critical points $\bar{X}\in C(\T_C ,\R )$ of the discrete $(\Q,\w)$-Lagrangian functional of order $m$ defined by \eqref{dislag} are solutions for all $i=1,\dots ,N-1$ of the following set of equations:
\begin{equation}
\label{euler1}
\left .
\begin{array}{l}
    \di\int_{t_{i-1}}^{t_i} \left ( 
    \di\frac{\partial L}{\partial x} (\star_t ) \l_{i-1,m} (t)  
    +\di\frac{\partial L}{\partial v} (\star_t ) \di\frac{d^+}{dt} (\l_{i-1 ,m}) (t)
    \right ) \, \Delta_{\w} t \\
    +
    \di\int_{t_i}^{t_{i+1}} \left ( 
    \di\frac{\partial L}{\partial x} (\star_t ) \l_{i,0} (t)  
    +\di\frac{\partial L}{\partial v} (\star_t ) \di\frac{d^+}{dt} (\l_{i,0}) (t)
    \right ) \, \Delta_{\w} t 
    =0 .
\end{array}
\right .
\end{equation}
and 
\begin{equation}
\label{euler2}
    \di\int_{t_i}^{t_{i+1}} \left ( 
    \di\frac{\partial L}{\partial x} (\star_t ) \l_{i,j} (t)  
    +\di\frac{\partial L}{\partial v} (\star_t ) \l'_{i,j} (t)  
    \right ) \, \Delta_{\w} t =0 ,
\end{equation}
for $j=1,\dots ,m-1$, where $\star_t :=(\bar{x}, \di\frac{d^+}{dt} (\bar{x} ))$ and $\l'_{i,j}$ is the derivative of $\l_{i,j}$.
\end{theorem}

Equations \eqref{euler1} are called the $(\Q,\w )$-{\it discrete Euler-Lagrange equations of order $m$} and correspond to the definition of {\it Galerkin variational integrators} (see \cite{marsden,leok1,leok2,ober}. The $(m-1)$ equations \eqref{euler2} give the constraints of the values of the control function $X_C$. 

\begin{proof}
By linearity of the embedding map, we have 
\begin{equation}
    \L (\bar{X}+\epsilon \bar{Y} ) =\di\int_a^b 
    L(\bar{x} +\epsilon \bar{y} , \di\frac{d^+}{dt} (\bar{x} ) +\epsilon \di\frac{d^+}{dt} (\bar{y} ) ) \, \Delta_{\w} t .
\end{equation}
As $L$ is $C^2$ with respect to all the variables, using a Taylor expansion, we obtain 
\begin{equation}
    L(\bar{x} +\epsilon \bar{y} , \di\frac{d^+}{dt} (\bar{x} ) +\epsilon \di\frac{d^+}{dt} (\bar{y} ) ) =L (\star_t ) ) +
    \epsilon \left ( 
    \di\frac{\partial L}{\partial x} (\star_t ) \bar{y} 
    +\di\frac{\partial L}{\partial v} (\star_t ) \di\frac{d^+}{dt} (\bar{y}) 
    \right ) 
    +o(\epsilon ) ,
\end{equation}

We then obtain, replacing in the functional and using the linearity of the $m$-th order anti-derivative,
\begin{equation}
    \L (\bar{X}+\epsilon \bar{Y} ) =\L (\bar{X} ) +\epsilon 
    \di\int_a^b \left ( 
    \di\frac{\partial L}{\partial x} (\star_t ) \bar{y} 
    +\di\frac{\partial L}{\partial v} (\star_t ) \di\frac{d^+}{dt} (\bar{y}) 
    \right ) \, \Delta_{\w} t +o (\epsilon ) . 
\end{equation}

As a consequence, we have 
\begin{equation}
    D\L (\bar{X} ) (\bar{Y} ) = \di\int_a^b \left ( 
    \di\frac{\partial L}{\partial x} (\star_t ) \bar{y} 
    +\di\frac{\partial L}{\partial v} (\star_t ) \di\frac{d^+}{dt} (\bar{y}) 
    \right ) \, \Delta_{\w} t . 
\end{equation}
A critical point $\bar{X}$ then satisfies for all $\bar{Y}\in C_{(0)} (\bar{\T} ,\R )$ the equality
\begin{equation}
    \di\int_a^b \left ( 
    \di\frac{\partial L}{\partial x} (\star_t ) \bar{y} 
    +\di\frac{\partial L}{\partial v} (\star_t ) \di\frac{d^+}{dt} (\bar{y}) 
    \right ) \, \Delta_{\w} t =0. 
\end{equation}
We conclude the proof using the discrete Du Bois-Reymond lemma \ref{dubois}.
\end{proof}

\section{Proof of the discrete Du Bois-Reymond lemma}
\label{proofdubois}
Using the decomposition \eqref{decomp}, we have $\bar{y} =y +y_{C}$ which leads to 
\begin{equation}
    \di\int_a^b \left ( 
    \bar{f}(t) y(t) 
    +\bar{g}(t) \di\frac{d^+}{dt} (y)(t) 
    \right ) \, \Delta_{\w } t 
    +
    \di\int_a^b \left ( 
    \bar{f}(t) y_C (t) 
    +\bar{g}(t) \di\frac{d^+}{dt} (y_C ) 
    \right ) \, \Delta_{\w} t =0 .
\end{equation}
By the independence of the variations $Y$ and $Y_C$, we then obtain 
\begin{equation}
\label{cont1}
    \di\int_a^b \left ( 
    \bar{f}(t) y(t) 
    +\bar{g}(t) \di\frac{d^+}{dt} (y)(t) 
    \right ) \, \Delta_{\w} t =0,
\end{equation}
for all $Y$ satisfying $Y(a)=Y(b)=0$ and 
\begin{equation}
\label{cont2}
    \di\int_a^b \left ( 
    \bar{f}(t) y_C (t) 
    +\bar{g}(t) \di\frac{d^+}{dt} (y_C ) 
    \right ) \, \Delta_{\w} t =0 ,
\end{equation}
for all $Y_C \in C((\T )_C ,\R )$.\\

For all $i\in \{1,\dots ,N-1\}$, and $j\in \{ 1,\dots ,m-1\}$, we choose a variation $\bar{Y}$ such that $Y=0$ and $Y_C (t)=0$ for all $t\in (\T )_C \setminus \{ c_{i,j} \}$ and $Y_C (c_{i,j} )\not=0$. The constraint \eqref{cont1} is satisfied and constraint \eqref{cont2} implies 
\begin{equation}
Y_C (c_{i,j} )   \left (  \di\int_{t_i}^{t_{i+1}} \left ( 
    \bar{f}(t) \l_{i,j} (t)  
    +\bar{g}(t) \l'_{i,j} (t)  
    \right ) \, \Delta_{\w} t \right ) =0 .
\end{equation}
Indeed, $y=0$ and $y_C$ is only non zero on $]t_i ,t_{i+1} [$ with $y_C (t)=Y_C (c_{i,j} ) l_{i,j} (t)$. As $Y_C (c_{i,j} )\not=0$, this implies \eqref{cond2}. \\

In the same way, for $i\in \{1,\dots ,N-1\}$, let $\bar{Y}$ be such that $Y_C =0$ on $(\T )_C$ and $Y(t)=0$ for all $t\in \T \setminus \{t_i\}$ and $Y(t_i )\not= 0$. The constraint \eqref{cont2} is directly satisfied and \eqref{cont1} gives
\begin{equation}
Y(t_i ) \left ( 
\left .
\begin{array}{l}
    \di\int_{t_{i-1}}^{t_i} \left ( 
    \bar{f} (t) \l_{i-1,m} (t)  
    +\bar{g} (t) \di\frac{d^+}{dt} (\l_{i-1 ,m}) 
    \right ) \, \Delta_{\w} t \\
    +
    \di\int_{t_i}^{t_{i+1}} \left ( 
    \bar{f} (t) \l_{i,0} (t)  
    +\bar{g}(t) \di\frac{d^+}{dt} (\l_{i,0}) 
    \right ) \, \Delta_{\w} t 
    =0
\end{array}
\right.
\right ) =0.
\end{equation}
Indeed, we have $y_C =0$ and $y(t)$ is equal to $Y(t_i) \l_{i,l} (t)$ for all $t\in [t_i ,t_{i+1}]$ and to $Y(t_i ) \l_{i,0} (t)$ for all $t\in [t_i ,t_{i+1} ]$. As $Y(t_i)\not=0$, this implies \eqref{cond1}. This concludes the proof.

\section{Conclusion}

The High-order differential and integral calculus allows us to give a presentation of Galerkin variational integrators in the framework of discrete embeddings. Following our previous work on the embedding of Lagrangian and Hamiltonian partial differential equations (PDEs) \cite{cgp}, we would like to extend this formalism to cover high-order methods for PDEs.


\begin{thebibliography}{15}
\bibitem{arnold}
V.~I. Arnold.
\newblock {\em Mathematical methods of classical mechanics}, volume~60 of {\em
  Graduate Texts in Mathematics}.
\newblock Springer-Verlag, New York, 1979.

\bibitem{cr1}
	J.~Cresson.
	\newblock {Introduction to embedding of Lagrangian systems}.
	\newblock {\em International Journal of Biomathematics and Biostatistics},
	1(1):23--31, 2010.
	
\bibitem{cr2}
J.~Cresson.
\newblock {Continuous versus discrete structures I - Discrete embeddings of ordinary differential equations and discrete Lagrangian systems}.
\newblock {\em Proceedings of the Institute of Mathematics and Mechanics}, 2025.
		
\bibitem{cgp}
J. Cresson, I. Greff, C. Pierre, Discrete embeddings for Lagrangian and Hamiltonian systems, Acta Math. Vietnam. 43 (2018), no. 3, 391-413

\bibitem{rouba} J. Cresson, R. Safi, Discrete embedding of Lagrangian/Hamiltonian systems and the Marsden-West approach to variational integrators-the order one case, Monografias matematicas "Garcia de galdeano" no. 43, p. 107-118, 2024.

\bibitem{rouba2} J. Cresson, R. Safi, Continuous versus discrete structures IV - Mid-point embedding of Hamiltonian systems and the Wendlandt-Marsden mechanical integrators, Arxiv arXiv:2211.16144. 

\bibitem{gelfand}
S.~V.~Fomin, I.~M.~Gelfand.
\newblock {\em Calculus of Variations}.
\newblock Dover Books on Mathematics. Dover, 2000.

\bibitem{hairer}
	E.~Hairer, C.~Lubich, and G.~Wanner.
	\newblock {\em {G}eometric numerical integration: structure-preserving
		algorithms for ordinary differential equations}, volume~31.
	\newblock Springer Science \& Business Media, 2006.

\bibitem{leok1} M. Leok, Generalized Galerkin Variational Integrators, arXiv : math/0508360. (2005).

\bibitem{leok2} M. Leok, T. Shingel, General techniques for constructing variational integrators. Frontiers of Mathematics in China, 7(2), 273-303, (2012).

\bibitem{marsden}
J.E. Marsden and M.~West.
\newblock {D}iscrete mechanics and variational integrators.
\newblock {\em Acta Numerica 2001}, 10:357--514, 2001.

\bibitem{ober} S. Ober-Bl\"obaum, M. Vermeeren, Superconvergence of Galerkin variational integrators, IFAC-PapersOnLine, Volume 54, Issue 19, 2021, Pages 327-333.
\end{thebibliography}
\end{document}